\documentclass[a4paper,10pt]{amsart}
\usepackage[utf8]{inputenc}
\usepackage{amssymb}
\usepackage{amsthm}
\usepackage{mathrsfs}
\newtheorem{thm}{Theorem}[section]
\newtheorem{prop}[thm]{Proposition}
\newtheorem{lemma}[thm]{Lemma}

\theoremstyle{definition}
\newtheorem*{defn}{Definition}

\newtheorem*{remark}{Remark}

\numberwithin{equation}{section}

\title{Symmetric seminorms and the Leibniz property}
\author{Zoltán Léka}

\address{Royal Holloway, University of London \\ Egham Hill \\ Egham \\ Surrey \\ TW20 0EX \\ United Kingdom}
\email{zoltan.leka@rhul.ac.uk}

\thanks{This study was supported by the Marie Curie IF Fellowship, Project 'Moments', Grant no. 653943 and by the Hungarian Scientific Research Fund (OTKA) grant
no. K104206.}


\begin{document}

\begin{abstract}
  We show that certain symmetric seminorms on $\mathbb{R}^n$ satisfy the Leibniz inequality. 
  As an application, we obtain that $L^p$ norms of centered bounded real functions, defined on probability spaces, have the same property. 
  Even though this is well-known for the standard deviation it seems that the complete result has never been established.
  In addition, we shall connect the results with the differential calculus introduced by Cipriani and Sauvageot and Rieffel's non-commutative Riemann metric.
\end{abstract}

\maketitle

\section{Introduction}
 Let $(S, \mathcal{F}, \mu)$ denote a probability space and let $1 \leq p < \infty.$ The seminorm given by the $p$th absolute central moment
 of a random variable $f \colon S\rightarrow \mathbb{R}$ is
  $$ \sigma_p(f; \mu) = \|f - \mathbb{E}f \|_p =\left(\int_S \left| f - \int_S f \: d\mu \right|^p \: d\mu \right)^{1/p}.$$
 One of the most used quantity in probability theory and statistics is the standard deviation (when $p=2$). 
 Recently M.A. Rieffel observed that the standard deviation in ordinary and non-commutative probability spaces satisfies the 
 strong Leibniz inequality and even matricial seminorms have the same property \cite{R2}. To be precise, we say
   that a seminorm $L$ on a unital normed algebra $(\mathcal{A}, \|\cdot \|)$ is strongly Leibniz if (i) $L(1_\mathcal{A}) = 0,$ (ii) the Leibniz property
 $$ L(ab) \leq \|a\|L(b) + \|b\| L(a) $$ holds for every $a, b \in \mathcal{A}$ and, furthermore, 
 (iii) for every invertible $a,$ $ L(a^{-1}) \leq \|a^{-1}\|^2 L(a)$ follows.
 For an ordinary probability space $(S, \mathcal{F}, \mu),$ this means that for every $f$ and $g \in L^\infty(S, \mu),$
 we have the inequalities 
   $$ \|fg-\mathbb{E}(fg)\|_2 \leq \|g\|_{\infty}\|f-\mathbb{E}f\|_2 + \|f\|_{\infty}\|g-\mathbb{E}g\|_2 $$
 and  
    $$  \|f^{-1}-\mathbb{E}(f^{-1})\|_2 \leq \|f^{-1}\|_\infty^2 \|f-\mathbb{E}f\|_2 \quad \mbox{ if } f^{-1} \in  L^\infty(S, \mu).  $$
 The study of strongly Leibniz seminorms regarded as non-commutative metrics on quantum metric spaces
 was initiated by M. Rieffel in his seminal papers \cite{R0c}, \cite{R1}, \cite{RL}. They played a crucial role in the development 
 of a quantum theory for the Gromov--Hausdorff distance. A quantized version of this theory was established
 in the recent papers by Li and Kerr \cite{KL}, W. Wu \cite{Wu}, and a thorough survey is \cite{La}.
 
 The most natural sources of strongly Leibniz seminorms are normed first-order differential calculi. We recall now that a normed
 first order differential calculus is a couple $(\Omega, \partial),$ where $\Omega$ is a normed bimodule over $\mathcal{A}$
 such that  $$ \| a \omega b\| \leq \|a\| \|\omega\|_\Omega \|b\| \mbox{for all } a, b \in \mathcal{A} \mbox{ and } \omega \in \Omega,$$ 
 and $\partial \colon \mathcal{A} \rightarrow \Omega$
 is a derivation which satisfies the Leibniz rule $\partial(ab)  = \partial(a)b + a\partial(b).$ Readily, 
 $$ L(a) = \|\partial a\|_\Omega$$ is a (strongly) Leibniz seminorm on $\mathcal{A}$ (see \cite[Proposition 1.1]{R2}).
 
  A prototype of Leibniz seminorms is the Lipschitz number
  $$ L_\rho(f) = \sup \{|f(x) - f(y)| / \rho(x,y) \colon {x \neq y} \}$$ of complex-valued continuous functions
  defined on any compact metric space $(X,\rho)$ (\cite[Proposition 1.5.3]{NW2}). Interestingly, one can obtain $L_\rho$ by means of
  a normed first order differential calculus (\cite[Proposition 8]{NW}, \cite[Example 11.5]{R1}).
  We direct the interested reader to \cite{AC}, \cite{R0}, \cite{NW} and \cite{NW2}
  for a comprehensive study of general Lipschitz seminorms, Lip-norms, and the associated Lipschitz algebras.
  Although, we are unaware
  of any characterization of the Leibniz property \cite[Question 6.3]{R0}, the lattice inequality $L(f \vee g) \leq L(f) \vee L(g),$ for all real $f$ and $g,$ is sufficient
  to conclude that a Lip-norm $L$ is Leibniz (\cite[Theorem 8.1]{R0}).

  It is important to notice that any symmetric Dirichlet form $(D(\mathscr{E}), \mathscr{E} )$ defined on a dense domain $D(\mathscr{E})$ 
  of the real Hilbert space $L^2(S, \mu)$ satisfies the Leibniz inequality, see e.g. \cite[Theorem 1.4.2]{Fuk}, \cite[Corollary 3.3.2]{BH}.
  See \cite{K} and \cite{Y} for Dirichlet forms on finite sets, graphs and fractals.
  Furthermore, F. Cipriani and J.--L. Sauvageot \cite{CS} showed that every regular $C^*$-Dirichlet form can be
  represented as a quadratic form associated to a derivation taking its values in a Hilbert module, which is a direct link
  to the Leibniz rule.
  
  Back to the standard deviation, one can present a direct simple proof of its strong Leibniz property (see \cite{R1}, \cite{BeL}). 
  More interestingly, the (quantum) standard deviation completely fits into the aforementioned machinery of differential $1$-forms. 
  This idea was exploited independently from \cite{CS} in some depth in Rieffel's papers \cite{R1}, \cite{R2} and based on the observation that the variance of any random 
  variable $f \in L^2(S, \mu)$ can be written as $$ \|f - \mathbb{E}f\|_2^2 = {1 \over 2}  \iint_{S \times S} |f(x) - f(y)|^2 \: d\mu(x ) \: d\mu(y)$$ 
  (that is, it is a Dirichlet form on $L^2(S, \mu)).$
  His differential calculus is defined through the concept of spectral triples, introduced by Alain Connes, and what 
  he calls non-commutative Riemann metric, see \cite{AC} and \cite{R1}. Hence one can say that the standard deviation, commutative or not, shares a flavor of Connes' noncommutative geometry.

  The main goal of this paper is to show that the Leibniz inequality
   $$ \|fg - \mathbb{E}fg\|_p \leq \|g\|_\infty \|f - \mathbb{E}f\|_p + \|f\|_\infty \|g - \mathbb{E}g\|_p $$
  is satisfied for all $1 \leq p \neq 2 < \infty$ and real $f,g \in L^\infty(S, \mu ),$ which does not seem to have been noticed previously.
  We remark that the end-point case $p=\infty$ has already been settled in the recent paper \cite{BeL}.
  First, we prove the result for the finite state space $S_n = \{1, \hdots, n\}$ endowed with the uniform probability measure.
  To get a friendly approach to the subject, we shall replace the $\ell^p$ norms with symmetric norms on $\mathbb{R}^n.$
  It should be stressed here that the essential part of the paper works with symmetric norms and the uniform case. 
  In Section 3 we shall investigate the results in terms of different differential calculi, including the 
  Cipriani--Sauvageot algebraic construction of differential $1$-forms, and a very brief connection with Rieffel's 
  non-commutative Riemann metric. In Section 4, the failure of the normed bimodule property will lead us to a finite dimensional example of
  a Leibniz seminorm that is not strongly Leibniz. Lastly, in Section 5, we shall derive the Leibniz inequality
  for arbitrary probability measures by applying our earlier results on Leibniz seminorms in probability spaces \cite{BeL}.
  Our paper is in part an attempt to reveal a possible link between normed differential calculi and absolute central moments of bounded functions (random variables). Our future plan is to study the corresponding results in non-commutative matrix and $C^*$-algebras as well as 
  the case of complex-valued functions.

   \section{Leibniz inequality for symmetric seminorms}  
  
  At first, we collect a few notations we require in order to prove the main results. 
  
  \subsection{Symmetric norms}

  We say that a norm $\|\cdot\|$ on $\mathbb{R}^n$ is symmetric if it is invariant under sign-changes and
   permutations of the components. Symmetric norms are monotone which means that  
      $$    \|x\| \leq \|y\| \quad \mbox{ if } \quad  |x|^\downarrow  \leq |y|^\downarrow, $$
      where $|x|^\downarrow$ denotes the usual non-increasing rearrangement of the vector $|x|.$  
    Furthermore, the norm $\|\cdot\|$ is absolute so 
      $$ \|x\| = \||x|\|$$ for every $x \in \mathbb{R}^n$
     (see \cite[Section 2]{B}). 
  
   The vector $k$-norms (or Ky Fan $k$-norms) are special examples of symmetric norms. Indeed, the vector $k$-norm of $x$ is defined by  
  $$ \|x\|_{(k)} = \sum_{i=1}^k |x_i|^\downarrow.$$
 In the case when $k = n$ and $k = 1,$ we obtain the usual $\ell^1$ and $\ell^\infty$ norms on $\mathbb{R}^n$, 
 denoted by $\|\cdot\|_1$ and $\|\cdot\|_\infty$, respectively. We recall now that
 the dual norm of any symmetric norm is symmetric as well. This follows easily from the duality relation $\|x\|_* = \max \{ (x,y) \colon \|y\| \leq 1\}. $
    
A celebrated theorem of Ky Fan says that, for any $x, y \in \mathbb{R}_+^n,$ the inequalities $$\|x\|_{(k)} \leq \|y\|_{(k)}$$ hold
for every $1 \leq k \leq n,$ that is, $x$ is weakly majorized by $y,$ if and only if $$\|x\| \leq \|y\|$$ for every symmetric norm $\|\cdot\|$ on $\mathbb{R}^n$ (see \cite{B} or \cite[Chapter 15]{BS}). Hence one 
can look upon the vector $k$-norms as the cornerstones of symmetric norms.

 Following Barry Simon's terminology in \cite[p. 248]{BS}, let us introduce a class of real matrices.
 
 \begin{defn}
    We say that a matrix $A \in M_n(\mathbb{R})$ is real substochastic if 
     \begin{eqnarray*}
      \begin{split}
         \sum_{i=1}^n |a_{ij}| &\leq 1, \quad j = 1, \hdots, n, \\
         \sum_{j=1}^n |a_{ij}| &\leq 1, \quad i = 1, \hdots, n.
      \end{split}
     \end{eqnarray*}
 \end{defn}
It is simple to see that $A$ is real substochastic if and only if $A$  is a contraction on $\mathbb{R}^n$ endowed with the $\ell^1$ norm and the $\ell^\infty$ norm; 
i.e. $\|Ay\|_1 \leq \|y\|_1$ and $\|Ay\|_\infty \leq \|y\|_\infty$ for all $y \in \mathbb{R}^n.$
Additionally, if one can guarantee a proper linear connection between the vectors $x$ and $y,$ i.e. $Ay = x$ for some $A \in M_n(\mathbb{R}),$
we can use interpolation methods. Actually, the Calderón--Mityagin theorem (see \cite{C}, \cite{Mi} or \cite[Theorem 15.17]{BS}) says that
if $A$ is a real substochastic matrix then $$ \|Ay\| \leq \|y\|$$ follows for all $y \in \mathbb{R}^n$ and symmetric norms $\|\cdot\|.$

  Let $x = (x_1, \hdots, x_n) \in \mathbb{R}^n.$ Let us introduce the symmetric matrix 
  $\Theta_x$ with zero row 
   and column sum defined by 
   $$(\Theta_x)_{ij}    = \begin{cases}
                                  {1 \over 2n} (x_i + x_j) & \mbox{if } i \neq j \\
                               - \sum_{k : k \neq i } (\Theta_x)_{ik}  &  \mbox{ if } i  = j. 
                              \end{cases}
    $$
   Let us define the matrix $$I_x =  I_n + \Theta_x,$$ where $I_n$ denote the $n \times n$ identity matrix.
  Throughout the section we shall use the notation $\mathbb{E} f = {1 \over n} \sum_{i=1}^n f_i {\bf 1},$ 
  where $f = (f_1, \hdots, f_n) \in \mathbb{R}^n$ and $\mathbf{1}$ stands for the 
  constant $1$ vector. 
  Moreover, we shall consistently use $f$ and $g$ for vectors of $\mathbb{R}^n$ and $fg$ for their pointwise product.

  Our first proposition links the product of two vectors $f, g \in \mathbb{R}^n$ with the 
  matrices $I_{f+{\bf 1}},$ $I_{g+{\bf 1}}.$  
  
 \begin{prop} For any $f,$ $g \in \mathbb{R}^n,$ 
  $$ I_{f+ {\bf 1}}(g-\mathbb{E}g) + I_{g+ {\bf 1}}(f - \mathbb{E}f) =  \mathbb{E}(fg) - fg.$$
  \end{prop}
  
\begin{proof}
  Clearly, it is enough to show that 
  $$ I_{f}(g-\mathbb{E}g) + I_{g}(f - \mathbb{E}f) =  \mathbb{E}((f- {\bf 1})(g- {\bf 1})) - (f- {\bf 1})(g- {\bf 1})$$
holds. A straightforward calculation gives for every index $1 \leq m \leq n$ that
    \begin{eqnarray*}
    \begin{split}
    n(I_{f}(g-\mathbb{E}g) +& I_{g}(f - \mathbb{E}f))_m \\ 
     &= {1 \over 2n} \sum_{1 \leq i \neq m \leq n}  \sum_{1 \leq j \leq n} (f_i+f_m)(g_i - g_j) \\
     &  \quad + \Bigl(1 - {1 \over 2n } \sum_{1 \leq i \neq m \leq n} (f_m + f_i) \Bigr) \sum_{1 \leq j \leq n} (g_m - g_j)\\
    &   \quad + {1 \over 2n} \sum_{1 \leq i \neq m \leq n}  \sum_{1 \leq j \leq n} (g_i+g_m)(f_i - f_j) \\
    &  \quad + \Bigl(1 - {1 \over 2n } \sum_{1 \leq i \neq m \leq n} (g_m + g_i) \Bigr) \sum_{1 \leq j \leq n} (f_m - f_j)\\
    &= {1 \over 2n} \sum_{1 \leq i \neq m \leq n}  (f_i+f_m) \Biggl( \sum_{1 \leq j \leq n}(g_i - g_j) - \sum_{1 \leq j  \leq n} (g_m - g_j) \Biggr) \\
    &  \quad + {1 \over 2n} \sum_{1 \leq i \neq m \leq n}  (g_i+g_m) \Biggl( \sum_{1 \leq j \leq n}(f_i - f_j) - \sum_{1 \leq j  \leq n} (f_m - f_j) \Biggr) \\
    &  \quad + \sum_{1 \leq i \leq n} (g_m - g_i)  +   \sum_{1 \leq i \leq n} (f_m - f_i)\\
    &= {1 \over 2} \sum_{1 \leq i  \leq n}  ((f_i+f_m) (g_i - g_m) + (g_i+g_m) (f_i - f_m)) \\ 
    & \quad + \sum_{1 \leq i \leq n} (g_m - g_i + f_m - f_i)\\ 
    &= \sum_{1 \leq i  \leq n}  (f_ig_i - f_m g_m + g_m - g_i + f_m - f_i) \\
    &= \left(\sum_{1 \leq i  \leq n}  (f_i - 1)(g_i-1)\right) - n(f_m-1)(g_m-1)\\
       &=  n(\mathbb{E}((f- {\bf 1})(g - {\bf 1})) - (f - {\bf 1})(g - {\bf 1}))_m,
    \end{split}
   \end{eqnarray*}
  which is what we intended to have.  
\end{proof}

   Let us remember that 
 the dual norm of the vector $k$-norm is
  $$ \|x\|_{(k)^*} = \max \left(\|x\|_\infty, {\|x\|_1 \over k} \right) \qquad x \in \mathbb{R}^n$$ (e.g. \cite[Ex. IV.2.12]{B}).

   Let $\mathfrak{B}_{(k)^*} = \{ x \in \mathbb{R}^n \colon \|x\|_{(k)^*} \leq 1\}$ denote the closed unit ball of the dual space $(\mathbb{R}^n,\|\cdot\|_{(k)})^*.$
 Then the set of extreme points of $\mathfrak{B}_{(k)^*} $ can be readily described. The result is well-known, but we sketch a short proof for the sake of completeness. 

\begin{lemma}
    $$ {\rm ext }  \: \mathfrak{B}_{(k)^*} = \left\{ \sum_{{\substack{i \in S}}} \pm  e_i \colon  S \subseteq \{1, \hdots, n\} \mbox{ and } |S| = k  \right\},$$
   where $e_i$-s denote the standard basis elements of $\mathbb{R}^n.$   
\end{lemma}

\begin{proof}
  Denote $\mathfrak{K}_0$ the points of the $n$-cube $[-1,1]^n$ which has at most $k$ non-zero coordinates. It is not difficult to see that
   $$ \mbox{ conv } \mathfrak{K}_0 = \mathfrak{B}_{(k)^*}. $$
   In fact, pick a point $v$ in $\mathfrak{B}_{(k)^*}$ which has at most $k+1$ non-zero coordinates. Denote $v_i$ a coordinate
   of $v$ which has the smallest non-zero modulus. Obviously, $|v_i| \leq 1.$ Now choose a vector $c \in \{-1,0,1\}^n$ such that
   the support of $c$ has cardinality $k,$ $i \in $ supp $c$ and sign $c_j = \mbox{sign } v_j$ for every $j \in \mbox{ supp } c.$ Then
    it is simple to see that 
    $$ {v - |v_i|c \over 1 - |v|_i} \in \mathfrak{B}_{(k)^*}. $$
    
    Iterating the previous process, we arrive a point which has at most $k$ non-zero coordinates. 
    This point is the convex combination of vertices of a proper $k$-cube in $[-1,1]^n.$ 
\end{proof}

Now we are ready to prove the following proposition.
   
\begin{prop}
  For every $f \in [-1,1]^n$ and $1 \leq k \leq n,$ the operator 
  \begin{eqnarray*}
   \begin{split}
    I_{f+{\bf 1}}^*  \colon (\mathbb{R}^n,\|\cdot\|_{(k)^*}) &\rightarrow (\mathbb{R}^n,\|\cdot\|_{(k)^*})/\mathbb{R}, \quad 
                                                  x &\mapsto I_{f+ {\bf 1}}x + \mathbb{R}
   \end{split} 
  \end{eqnarray*}
 is a contraction.
  
\end{prop}

\begin{proof}
 First, to get an upper bound on the norm of $I_{f+{\bf 1}}^*,$ it is enough to calculate the norm of the class $I_{f+{\bf 1}}v$
 for every extreme point $v$ of the unit ball $(\mathbb{R}^n,\|\cdot\|_{(k)^*}).$ From Lemma 2.2, we can assume that 
 $$v = \sum_{i \in S_+} e_i - \sum_{i \in S_-} e_i$$ for some disjoint sets $S_+, S_- \subseteq \mathbb{Z}_n$ such that 
 $|S_-| + |S_+| = k.$ For any $x, y \in \mathbb{R}^n$ and  $0 \leq s \leq 1,$ 
 we have $I^*_{sx + (1-s)y} = sI^*_x + (1-s)I^*_y.$ Furthermore, since the quotient norm is convex, one has
   \begin{eqnarray*}
    \begin{split}
      \|I_{f+{\bf 1}}v\|_{(k)^*} &= \min_{\lambda \in \mathbb{R}} \|I_{f+{\bf 1}}v - \lambda {\bf 1}\|_{(k)^*} \\
      &\leq \max_{x \in [0,2]^n} \min_{\lambda \in \mathbb{R}} \|I_{x}v - \lambda {\bf 1}\|_{(k)^*}  \\
      &=  \max_{x \in \{0,2\}^n} \min_{\lambda \in \mathbb{R}}  \|I_{x}v - \lambda {\bf 1}\|_{(k)^*} .
    \end{split} 
   \end{eqnarray*}  
 
  Next, pick an $x \in \{0,2\}^n.$ Set $$r_v = {1 \over n} \langle x,v \rangle .  $$
   In order to prove that  $I_x v$ is in the unit ball of the quotient space, it is enough to show that
   $$     \left\|I_{x}v - r_v 1 \right\|_{(k)^*} \leq 1.$$
  In fact, 
       \begin{eqnarray*}
    \begin{split} 
      \|I_xv - r_v 1 \|_\infty &= \max_{1 \leq i \leq n} \left|\left\langle I_x e_i - n^{-1} x, v\right\rangle\right|   \\
      &\leq \max_{1 \leq i \leq n} \left\| (I_x - n^{-1} x \otimes 1)e_i \right\|_{(k)} \|v\|_{(k)^*} \\
      &\leq \max_{1 \leq i \leq n} \left\| (I_x - n^{-1} x \otimes 1)e_i \right\|_{1}. \\
 \end{split} 
   \end{eqnarray*} 
   Let $s = \mbox{card} \{ i : x_i = 2\}.$ For any $1 \leq i \leq n,$ note that 
        \begin{eqnarray*}
    \begin{split} 
      \left\| (I_x - n^{-1} x \otimes 1)e_i  \right\|_{1} &=  \Biggl|1 - {1 \over 2n}\sum_{j = 1}^n (x_i + x_j) \Biggr| + {1 \over 2n} \sum_{j=1}^n |x_i - x_j| \\
                                                       &= \begin{cases}
                                                              {s\over n} + {n-s\over n} & \mbox{ if } x_i = 2, \\
                                                             \left(1 - {s \over n}\right) + {s \over n} &\mbox{ if } x_i = 0 \\
                                                          \end{cases} \\
                                                       &= 1.   
 \end{split} 
   \end{eqnarray*}  
   Thus $$ \|I_xv - r_v 1 \|_\infty   \leq 1. $$ 
   
   Now, let $P_S$ denote the projection $\sum_{i=1}^n x_ie_i \mapsto \sum_{i \in S} x_ie_i$ on $\mathbb{R}^n,$ where
  $S = S_- \cup S_+$ is the support of $v.$ Then
   \begin{eqnarray*}
    \begin{split} 
      \|I_xv - r_v 1 \|_{1} =& \sum_{i=1}^n \left|\left\langle P_S \left(I_x e_i - {1 \over n }x \right), v\right\rangle\right|  \\     
              &\leq  \sum_{i=1}^n \left\| P_S \left(I_x e_i - {1 \over n }x \right) \right\|_{(k)} \|v\|_{(k)^*} \\
            &\leq  \sum_{i=1}^n \left\| P_S \left(I_x e_i - {1 \over n }x \right) \right\|_{1} \\
            &= \sum_{i \in S} \left( \left|1 - {1 \over 2n}\sum_{j = 1}^n (x_i + x_j) \right|  + {1 \over 2n} \sum_{j \in S} |x_i - x_j| \right) \\
            &\quad +  \sum_{i \not \in S} {1 \over 2n} \sum_{j \in S} |x_i - x_j| \\
            &=  \sum_{i \in S}  \left(\left|1 - {1 \over 2n}\sum_{j = 1}^n (x_i + x_j) \right|  + {1 \over 2n} \sum_{j =1}^n |x_i - x_j|\right),\\
 \end{split} 
   \end{eqnarray*}  
   that is, 
     \begin{eqnarray*}
    \begin{split} 
      \|I_xv - r_v 1 \|_{1} &\leq \sum_{i \in S} \left\| (I_x - n^{-1} x \otimes 1)e_i  \right\|_{1} \\
                                  &= |S|.
      \end{split} 
   \end{eqnarray*}
  Hence
     $$  \|I_xv - r_v 1 \|_{(k)^*} \leq 1,  $$
   and the proof is complete.  
\end{proof}

  Let us define the hyperplane 
  $$ \mathfrak{X}_0 := \{ x \in \mathbb{R}^n \colon \mathbb{E}x = \sum_{i=1}^n x_i = 0 \} \subseteq \mathbb{R}^n.$$ 
 Obviously, the dual of the Banach space $(\mathfrak{X}_0,  \|\cdot\|_{(k)})$ is the quotient space $(\mathbb{R}^n,  \|\cdot\|_{(k)^*})/\mathbb{R}.$
 In fact, $\mathfrak{X}_0$ is a one co-dimensional subspace of $\mathbb{R},$ whilst $\langle y, x-\mathbb{E}x \rangle = 0$ holds for every
 $y \in \mathbb{R} {\bf 1}.$ Clearly, $I_{f+{\bf 1}} {\bf 1} = {\bf 1}.$ 
 Hence the  adjoint of $I_{f+{\bf 1}} \colon (\mathfrak{X}_0, \|\cdot\|_{(k)}) \rightarrow (\mathbb{R}^n, \|\cdot\|_{(k)})$ is the operator
   $$I^*_{f+{\bf 1}} \; \colon \; (\mathbb{R}^n,\|\cdot\|_{(k)^*}) \rightarrow (\mathbb{R}^n,\|\cdot\|_{(k)^*})/\mathbb{R}, \quad  x \mapsto I_{f+{\bf 1}}x + \mathbb{R}, $$
   of Proposition 2.3.
  Since  $\|I_{f+{\bf 1}}|\mathfrak{X}_0\| = \|(I_{f+{\bf 1}}|\mathfrak{X}_0)^*\|$ (see e.g. \cite[Proposition 2.3.10]{Pe}), 
  a straightforward corollary is the following statement. 
  
\begin{prop} For every $f \in [-1,1]^n,$ the operator $I_{f+{\bf 1}}$ is a contraction on the normed space $(\mathfrak{X}_0,\|\cdot\|_{(k)}).$
\end{prop}

 Furthermore, this leads us to the next proposition.
   
\begin{prop} For every symmetric $\|\cdot\|$ on $\mathbb{R}^n$ and $f \in [-1,1]^n$, $I_{f+{\bf 1}}$ is a contraction on $(\mathfrak{X}_0,\|\cdot\|).$
\end{prop}  
          
 \begin{proof}
    For every $x \in \mathfrak{X}_0$ and $1 \leq k \leq n,$ Proposition 2.4 says that 
       $$ \sum_{i=1}^k |I_{f+{\bf 1}}x|_i^\downarrow \leq \sum_{i=1}^k |x|_i^\downarrow.$$ 
    Thus the vector $|I_{f+{\bf 1}}x|$ is weakly majorized by $|x|.$ Now the absolute property of $\|\cdot\|$ and Ky Fan's theorem for symmetric norms give that   
     $$ \|I_{f+{\bf 1}}x\| = \||I_{f+{\bf 1}}x|\| \leq \||x|\| = \|x\|, $$
    which is what we intended to have. 
 \end{proof}
 
  Now one can readily prove the following Leibniz inequality for symmetric norms.  

\begin{thm}
  Let $\|\cdot\|$ be a symmetric norm on $\mathbb{R}^n.$ For every $f, g \in \mathbb{R}^n,$ we have 
    $$ \|fg-\mathbb{E}(fg)\| \leq \|g\|_{\infty}\|f-\mathbb{E}f\| + \|f\|_{\infty}\|g-\mathbb{E}g\|. $$
\end{thm}
 
\begin{proof}
   Without loss of generality, we can assume that $\|f\|_\infty = \|g\|_\infty = 1.$ 
   Applying Proposition 2.1 and Proposition 2.5, it follows that  
    \begin{eqnarray*}
    \begin{split}
    \|fg - \mathbb{E}(fg)\| &= \|I_{f+{\bf 1}}(g-\mathbb{E}g) + I_{g+{\bf 1}} (f-\mathbb{E}f)\|  \\
                    &\leq \|I_{f+{\bf 1}}|\mathfrak{X}_0\| \|g - \mathbb{E}g\| + \|I_{g+{\bf 1}}|\mathfrak{X}_0\| \|f - \mathbb{E}f\| \\
     &= \|g - \mathbb{E}g\| + \|f - \mathbb{E}f\|,
   \end{split} 
   \end{eqnarray*}
   and the proof is complete.
\end{proof}

 \begin{remark} One can give a direct proof of Proposition 2.5 via the Calderón--Mityagin interpolation result as we briefly indicate. For an $x \in [0,2]^n,$ let us consider the matrix
 $$ L_x = I_x - { 1  \over n} x \otimes {\bf 1}.$$
 
 We note that the off-diagonal part of $L_x$ is skew-symmetric: $(L_x)_{i,j} = - (L_x)_{j,i}$ for every $i \neq j,$
 hence $\|L_x^T\|_{1 \rightarrow 1} = \|L_x^T\|_{\infty \rightarrow \infty}.$
 From the proof of Proposition 2.3, it follows that 
 $$ \|L_x^T\|_{1 \rightarrow 1} \leq 1 \quad \mbox{ and } \quad \|L_x^T\|_{\infty \rightarrow \infty} \leq 1.$$
 Moreover, for any symmetric norm $\|\cdot\|,$ the adjoint of $I_x \colon (\mathfrak{X}_0, \|\cdot\|) \rightarrow (\mathbb{R}^n, \|\cdot\|), \: v \mapsto I_{x}v,$ is the operator
   $$I_x^* \colon (\mathbb{R}^n, \|\cdot\|_*) \rightarrow (\mathbb{R}^n, \|\cdot\|_*) / \mathbb{R},$$ where 
    $$ I_{x}^*v = I_{x}v + \lambda {\bf 1}$$
   and $\|\cdot\|
_*$ denotes the dual norm. Again, for any $v \in \mathbb{R}^n$, let $\displaystyle r_v = {1 \over n}\langle x, v\rangle.$ Then
   \begin{eqnarray*}
    \begin{split}
      \|I_{x}v  - r_v 1\|_* &=  \|I_{x}v  - {1 \over n}\langle x, v\rangle \|_*  \\
                                  &=  \|\langle  (I_{x} - {1 \over n} x \otimes 1)e_i, v\rangle_i \|_* \\
                                  &= \|L_{x}^T v\|_*.
    \end{split}
   \end{eqnarray*} 
   Since the dual norm $\|\cdot\|_*$ is symmetric, the Calderón--Mityagin theorem says that 
    $$ \min_{\lambda \in \mathbb{R} } \|I_{x}v  - \lambda {\bf 1} \|_* \leq \|L_{x}^T v\|_* \leq \|v\|_*.$$
   That is, 
    $$ \|I_{x}^*\| \leq 1,$$
   and the operator $I_{x}$  is a contraction on $(\mathfrak{X}_0, \|\cdot\|).$ 
 \end{remark}
 
 \begin{remark}  Perhaps it is appropriate to note that if $x  \in [0,1]^n$ then $I_x$ is doubly stochastic.
  Hence, the Birkhoff--von Neumann theorem gives that $\|I_x\|_{\|\cdot\| \rightarrow \|\cdot\|} \leq 1$ for any permutation invariant norm $\|\cdot\|$ on $\mathbb{R}^n.$
  Now assume that $f, g$ are nonnegative and $\|f\|_\infty = \|g\|_\infty = 1$  Then
    $$ I_{-f+{\bf 1}}( \mathbb{E}g - g) + I_{-g+{\bf 1}}( \mathbb{E}f -f) =  \mathbb{E}(fg) - fg,$$
  and the matrices $I_{-f+{\bf 1}},$ $I_{-g+{\bf 1}}$ are doubly stochastic as well. A simple corollary is the following statement.
\begin{thm}
  Let $\| \cdot\|$ be a permutation invariant norm on $\mathbb{R}^n.$  
  For any nonnegative vectors $f$ and $g$ in $\mathbb{R}_+^n,$  we have
 $$ \|fg-\mathbb{E}(fg)\|\leq \|g\|_{\infty}\|f-\mathbb{E}f\| + \|f\|_{\infty}\|g-\mathbb{E}g\|. $$
\end{thm}  
 \end{remark}

 \section{Derivations and the Leibniz inequality} 

  To have a description of the Leibniz inequality in terms of derivations, we shall need to introduce the  
  fundamental concepts of Laplacians and related Dirichlet forms on finite sets \cite{K}.
  
  \subsection{Laplacians and Dirichlet forms}
  We recall that a Laplacian matrix $\Delta$ is a non-positive definite matrix 
  such that its kernel is the subspace $\mathbb{R} {\bf 1}$ and all of its off-diagonals are non-negative.
  Let us remember that every Laplacian $\Delta$ determines a Dirichlet form $\mathscr{E}_\Delta(u,v) = - \langle u, \Delta v \rangle $ 
  on $\mathbb{R}^n \times \mathbb{R}^n.$ To be precise, for any $f \in \mathbb{R}^n$ let us define the vector
  $$ \overline{f}_i = \begin{cases}
                        0  &\mbox{ if }  f_i \leq 0, \\
                        f_i &\mbox{ if } 0 <  f_i < 1,  \\
                        1 &\mbox{ if } 1 \leq  f_i.
                        \end{cases} $$
   A symmetric bilinear form $\mathscr{E}$ is a Dirichlet form if it satisfies the following properties:
  \begin{itemize}
   \item[(i)] $\mathscr{E}(f,f) \geq 0,$
   \item[(ii)] $\mathscr{E}(f,f) = 0$ if and only if $f \in \mathbb{R} {\bf 1},$
   \item[(iii)] $\mathscr{E}( \overline{f},  \overline{f}) \leq \mathscr{E}(f,f)$  {(Markovian property)}.
  \end{itemize}

  Actually, there is a one-to-one correspondence between the Laplacians and the Dirichlet forms on finite sets,
  see \cite[Proposition 2.1.3]{K}.

  On the other hand, it is simple to see that
  \begin{eqnarray*} 
   \begin{split}
   |f_i g_i - f_j g_j| &= |f_i g_i - f_j g_i + f_j g_i - f_j g_j| \\
                       &\leq \|g\|_\infty |f_i - f_j| + \|f\|_\infty |g_i - g_j|,
   \end{split}
  \end{eqnarray*}
  hence the Leibniz inequality
   $${\mathscr{E}^{1/2}(fg,fg )}\leq \|g\|_\infty {\mathscr{E}^{1/2}(f,f)} + \|f\|_\infty {\mathscr{E}^{1/2}(g,g)}$$
  follows immediately (see \cite[p. 281]{Y}). 
 
  Now let $f \in \mathbb{R}^n$ and $P$ denote the orthogonal projection $f \mapsto {1 \over n} \sum_{i=1}^n f_i \mathbf{1}$ 
  with respect to the usual inner product on $\mathbb{R}^n.$ Then the operator
  $$ \Delta_u =  P-I_n = {1 \over n} \begin{pmatrix} 1-n & 1 & \hdots &  1 \cr 
                                      1  & 1-n&  &       1 \cr
                                    \vdots&  & \ddots & \vdots \cr
                                       1 &  \hdots &  1 & 1-n \end{pmatrix}$$ 
  is a Laplacian and 
           $$ \mathscr{E}_{\Delta_u}(f,f) = \|f - \mathbb{E}f\|_2^2,$$
  where $\|\cdot\|_2$ denotes the usual $\ell^2_n$-norm.
 
 
  \subsection{Derivations and the Leibniz inequality}
 
  Let $\mu$ be in general a probability measure on the set $S_n.$ 
  Then the variance of any random vector $f$ can be written as the Dirichlet form 
  $$ \sigma_2^2(f; \mu) = - \langle f, \Delta_\mu f \rangle = {1 \over 2} \sum_{x , y \in S_n} (f(x) - f(y))^2 \mu(x)\mu(y),$$
  where the off-diagonal part of the Laplacian $\Delta_\mu$ is $(\Delta_\mu)_{i,j} = \mu(i) \mu(j),$ $1 \leq i \neq j \leq n.$
  Then the deviation  $\displaystyle \sigma_2(f; \mu)$ can be represented as 
  the $L^2$-norm of the gradient vector $\partial_u f,$ where $\partial_u$ is the universal
  derivation $\partial_u f = f \otimes 1 - 1 \otimes f,$ in the Hilbert space 
  $L^2(S_n \times S_n , \mu \otimes \mu).$
      
  To obtain $\sigma_p$ as a norm of a derivation we need a refined approach. Let us consider the matrix algebra $M_n(\mathbb{R}) = \ell_n^\infty \otimes \ell_n^\infty$ endowed with 
  the Hilbert--Schmidt inner product as a bimodule over the finite dimensional algebra $\ell_n^\infty,$
  where the left and right actions are defined by linearity from 
   $$ a (b \otimes c) d = ab \otimes cd.$$
  Define the derivation $\partial \colon \ell_n^\infty \rightarrow M_n(\mathbb{R})$ by 
    \begin{eqnarray}
      \partial f = {1 \over \sqrt{2n}} (f \otimes 1 - 1 \otimes f),
     \end{eqnarray}
    which satisfies the Leibniz equality, i.e. $\partial( fg) = \partial f \cdot g + f \cdot \partial g.$
  The adjoint operator $\partial^*,$ defined by $ {\rm Tr}(A^T \partial f)= \langle \partial^* A, f \rangle$ for any $A \in M_n(\mathbb{R}),$ is the operator
   \begin{eqnarray}
      (\partial^* A)_i = -{1 \over \sqrt{2n} } (A(1 \otimes 1) - (1 \otimes 1)A)_{ii}.
   \end{eqnarray}
  Indeed, let  $\iota$  denote the canonical embedding of the algebra $\ell_n^{\infty}$ into $M_n(\mathbb{R})$ as the
  diagonal algebra.
  Then one has that $$\partial f  = {1 \over \sqrt{2n}}( (\iota f ) 1 \otimes 1 - 1 \otimes 1 (\iota f)). $$ 
  Since the extended derivation
    $d \colon A \mapsto A(1 \otimes 1) - (1 \otimes 1)A$ is a skew adjoint map on $M_n(\mathbb{R})$ with respect to the Hilbert--Schmidt inner product, we get  $\partial^* = -\iota^* d.$ 
 
     An elementary calculation implies the following lemma, whence we omit its proof. 
     
     \begin{lemma} One has the decomposition 
                           $$ -\Delta_u = \partial^* \partial.$$ 
     \end{lemma}

 
  Then the following definition is quite natural. 
  
  \begin{defn}
    Fix a symmetric norm $\|\cdot\|$ on $\mathbb{R}^n$ and let $\|\cdot\|_*$ denote its dual norm. We define a seminorm on the matrix algebra $M_n(\mathbb{R})$ by
   $$ \|A\|_\partial = \max \: \{ \mbox{ Tr}(A^T \partial f) \colon \|f\|_* \leq 1 \}.$$   
  \end{defn}

  The next proposition links the differential calculus $(M_n(\mathbb{R}), \partial)$ over $\ell^{\infty}_n$
  with the norms of centered vectors.
 
 \begin{prop} 
    Let $f=(f_1, \hdots, f_n) \in \ell_n^\infty$ and $\|\cdot\|$ be a symmetric norm on $\mathbb{R}^n.$ Then the equality
     $$ \Bigl\|f - {1 \over n} \sum_{i=1}^n f_i {\bf 1} \Bigr\| = \|\partial f \|_{\partial}$$ holds. 
  \end{prop}

   \begin{proof}
    From Lemma 3.1 and duality
   \begin{eqnarray*}
     \begin{split}
       \|\Delta_u f\| = \|\partial^*\partial f\| =& \max \left\{ \langle \partial^*\partial  f,g \rangle \colon \|g\|_* \leq 1 \right\} \\
                      =& \max \: \{ \mbox{Tr}(\partial f^T \partial g) \:  \colon \|g\|_* \leq 1 \} \\
                      =&  \|\partial f \|_{\partial}.
     \end{split}
   \end{eqnarray*}    
   \end{proof}
  
  The next theorem shows a certain module property of the seminorm $\|\cdot\|_\partial.$
   
   \begin{thm}
    For any $f$ and $g \in \mathbb{R}^n,$
       $$ \|\partial f \cdot g\|_\partial \leq \|g\|_\infty \|\partial f\|_\partial,$$    
     \vspace{.1cm}  
       $$ \|g \cdot \partial f\|_\partial \leq \|g\|_\infty \|\partial f\|_\partial. $$
   \end{thm}
  
   \begin{proof}
    First, we have 
    \begin{eqnarray*}
     \begin{split}
        (\partial f \cdot g)_{ij} = g_{j} (\partial f)_{ij} \quad \mbox{ and } \quad (g \cdot \partial f)_{ij} =  g_i (\partial f)_{ij} .
     \end{split}
   \end{eqnarray*} 
   For any $a \in \mathbb{R}^n,$ note that $I_{g} + {a \otimes \bf 1} = I_{g}$ holds on the subspace $\mathfrak{X}_0 = (I_n-P)\mathbb{R}^n.$ Thus
    $$  I_{g + \bf 1} = I_{g + \bf 1} - {{\bf 1} \over n} \otimes 1 = \Theta_g \quad \mbox{ on } \quad \mathfrak{X}_0.$$ 
   From Proposition 2.5
     $$ \|\Theta_gh\| \leq \|g\|_\infty\|h\| \quad \mbox{for any }h \in \mathfrak{X}_0 . $$
   On the other hand, a direct calculation shows that 
   \begin{eqnarray*}
     \begin{split}
       \langle h, \Theta_gf \rangle &=  {1 \over 2n} \sum_{i,j = 1}^n (f_i - f_j)(g_i+g_j)h_i = {1 \over 2n}   \sum_{i,j = 1}^n  (f_i - f_j) g_i (h_i - h_j) \\
                          &= \mbox{Tr}((g \cdot\partial f)^T \partial h) .                     
     \end{split}
    \end{eqnarray*}
   Thus
    \begin{eqnarray*}
     \begin{split}
     \|g \cdot \partial f\|_\partial &=  \max \{ \langle h, \Theta_gf \rangle \colon \|h\|_* \leq 1 \}\\
                                    &=  \|\Theta_gf\| \\
                                    &=  \|\Theta_g((I-P)f \oplus Pf)\|  \\
                                    &=  \|\Theta_g(I-P)f\| \\
                                    &\leq \|g\|_\infty \|(I-P)f\| \\
                                    &= \|g\|_\infty \|\partial f\|_\partial.
    \end{split}
   \end{eqnarray*}    
    The same argument gives that $\|\partial f \cdot g\|_{\partial} \leq \|g\|_\infty \|\partial f\|_{\partial}$ holds, hence
    the proof is complete. 
   \end{proof}

  We saw in Theorem 2.6 that the Leibniz inequality holds with symmetric norms. 
  Now one can provide a transparent reformulation of the proof relying upon the previous results. 
  
  \begin{proof}[Proof of Theorem 2.6]
   \begin{eqnarray*}
    \begin{split}
     \left\|fg - {1 \over n} \sum_{i=1}^n f_ig_i {\bf 1} \right\| = \|\partial^* \partial (fg)\| &= \|\partial^* (\partial f \cdot g) + \partial^* (f \cdot \partial g\| \\
                       &\leq \| \partial f \cdot g\|_\partial + \|f \cdot \partial g\|_\partial \\
                       &\leq  \|g\|_\infty \|\partial f\|_\partial + \|f\|_\infty \| \partial g\|_\partial \\
                       &=  \|g\|_\infty \|f - {1 \over n} \sum_{i=1}^n f_i {\bf 1} \| + \|f\|_\infty \|g - {1 \over n} \sum_{i=1}^n g_i {\bf 1} \|.                       
    \end{split}
   \end{eqnarray*}
  \end{proof} 
   
  At this point one can ask if the inequality 
       \begin{equation}
         \|f (\partial  g)  h\|_\partial \leq \|f\|_\infty  \|h\|_\infty \|\partial g\|_\partial    
       \end{equation}
  holds for every $f,g$ and $h \in \mathbb{R}^n$ and a normed 
  bimodule structure might appear on a certain subspace of $M_n(\mathbb{R}).$
  Then the strong Leibniz inequality would be an immediate corollary of (3.3) due to the derivation rule 
  $\partial f^{-1} = - f^{-1} (\partial f) f^{-1}.$ Unfortunately, this is rarely the case as we can now see. 
  Let us consider the seminorm $\|\cdot\|_{1,\partial} \equiv \|\partial^* \cdot\|_{1} $ on $M_n(\mathbb{R}).$ 
  Then, with choice of the vectors $f = h =(1,-1,1,1,1)$ and $g = (1,-1,0,0,0)$ in $\mathbb{R}^5,$ one has
    $$\| \partial g\|_{1,\partial} < \|f (\partial g) h\|_{1,\partial}.$$ 
  However, in the case of $\|\cdot\|_{2,\partial}$ seminorm one can apply the differential calculus invented by Cipriani and Sauvageot \cite{CS} in order to prove (3.3) (see Proposition (3.6) below). 
  
  Interestingly, our 
  numerical experiences support the conjecture that all the above seminorms in Theorem 2.6 are strongly Leibniz but we shall leave open this question. 
  
  The crucial point in the previous proof of Theorem 2.6 and in the failure of (3.3) is the decomposition $-\Delta_u = \partial^* \partial$ 
  which depends heavily on the Hilbert--Schmidt inner product. 
  Another decomposition would emerge from the aforementioned Hilbert bimodule structure on $M_n(\mathbb{R})$ used by Cipriani and Sauvageot.
  This is the content of the next section.

  \subsubsection{Cipriani--Sauvageot differential calculus} Here we shall briefly 
  describe the Hilbert bimodule structure introduced in \cite{CS} (the interested reader might see \cite{JLS} as well). 
  The motivation there was to prove that any regular $C^*$-Dirichlet form can be represented as a
  quadratic form associated to a closable derivation.
  To have a natural connection with the previous sections of the paper, we shall only describe 
  their algebraic construction in the real finite dimensional case. 
  
   Let us consider the left and right actions of $\ell^\infty_n$ 
   on $\ell^\infty_n \otimes \ell^\infty_n = M_n(\mathbb{R})$ by linearity from 
    \begin{eqnarray*}
     \begin{split}
         a(b \otimes c) = ab & \otimes c - a \otimes bc \\ 
         (b \otimes c)d &= b \otimes cd .  
     \end{split} 
    \end{eqnarray*}
     
   Let $\mathscr{E}$ be a Dirichlet form  on the set $S_n = \{1, \hdots, n\}.$
   Then a positive bilinear form on  $M_n(\mathbb{R})$ is given by 
   $$ (c \otimes d, a \otimes b )_{\mathcal{H}} 
   = {1 \over 2} (\mathscr{E}(c, abd) + \mathscr{E}(cdb, a)- \mathscr{E}(db, ca) ).$$
   The Hilbert space $\mathcal{H}$ is obtained by taking the factor space by the zero-norm subspace and then the completion.
   As a result, we get $\mathcal{H}$ is a Hilbert bimodule over $\ell^\infty_n$ \cite[Theorem 3.7]{CS}. Furthermore, the map 
   $\partial_0 \colon \ell^\infty_n \rightarrow \mathcal{H}$ defined by $$ \partial_0 f = f \otimes 1$$ 
   is a derivation on $\ell^\infty_n.$ Indeed, one can easily see that the Leibniz equality 
   $\partial_0(fg) = \partial_0 f \cdot g + f\cdot \partial_0 g$ is satisfied.
   Interestingly, from \cite[Theorem 4.7]{CS} one has the equality  
   \begin{equation}
    \mathscr{E}(f,f) = \|\partial_0 f\|_{\mathcal{H}}^2.    
   \end{equation}
   We know that
   that there is a one-to-one correspondence between Dirichlet forms on finite sets and Laplace matrices,
   hence every Laplace matrix $\Delta$ can be decomposed as 
   \begin{equation}
    - \Delta = \partial^*_0\partial_0,    
   \end{equation}
   where $\partial^*_0$ is the adjoint given by the formula
   $(\partial_0 f, a \otimes b)_{\mathcal{H}} = \langle f, \partial_0^* (a \otimes b) \rangle$ (see \cite[Theorem 8.2]{CS} for the general case).
      
   Now let us consider the Dirichlet form $\mathscr{E}_{\Delta_u}$ determined by $-\Delta_u = I-P.$
   First, let us calculate $\partial_0^*$ with respect to the inner product defined by $\mathscr{E}_{\Delta_u}.$
   
   \begin{lemma}
      The adjoint of the operator $\partial_0 \colon \ell^\infty_n \rightarrow \mathcal{H}$ is the linear map
        $$(\partial_0^* (a \otimes b))_i = {1 \over 2n} \sum_{j=1}^n (a_i - a_j) (b_i + b_j), \qquad 1 \leq i \leq n. $$
   \end{lemma}

    \begin{proof}
      A little computation shows that 
      \begin{eqnarray*}
        \begin{split}
          \langle \partial_0^* (a \otimes b),  c \rangle &= (a \otimes b, c \otimes 1)_{\mathcal{H}} \\ 
                                           &= \mathscr{E}_{\Delta_u} (c, ab) + \mathscr{E}_{\Delta_u} (a, bc) - \mathscr{E}_{\Delta_u} (b, ac)  \\
                                           &= {1 \over 2n} \sum_{i,j =1}^n (a_ib_i(c_i-c_j) + b_ic_i(a_i-a_j) - a_ic_i(b_i - b_j)) \\
                                           &= {1 \over 2n} \sum_{i=1}^n c_i \left( \sum_{j=1}^n (a_ib_i - b_ia_j + a_ib_j - a_jb_j) \right)\\
        \end{split}
      \end{eqnarray*}
     hence the proof is complete.
    \end{proof}

    Notice that we have $\partial_0^* (a \otimes b) = \Theta_b a$ with the notations of Section 2. 
    
   \begin{lemma}
      For any $f,g$ and $h \in \ell^\infty_n,$ 
      $$(\partial_0^* (f (\partial_0 g) h))_i = {1 \over 2n} \sum_{j=1}^n (g_i - g_j) (f_ih_j + f_jh_i).$$ 
   \end{lemma}
   
   \begin{proof}
     From Lemma 3.4 notice that 
       \begin{eqnarray*}
        \begin{split}
       \partial_0^* (f(\partial_0 g) h) = \partial_0^*(fg \otimes h - f \otimes gh)  &= \partial_0^* (f_ih_j(g_i - g_j))_{i,j} \\
       &= \left({1 \over 2n} \sum_{j=1}^n (f_ih_j + f_jh_i)(g_i-g_j)\right)_{1 \leq i \leq n},
        \end{split}
        \end{eqnarray*}
     which is what we intended to have.  
   \end{proof}
 
  The previous lemmas show that the adjoint operators $\partial_0^*$ and $(2n)^{-1/2}\partial^*$ in (3.2) are the same on the subspace of matrices with zero diagonal and
   \begin{equation}
      \partial_0^* (f (\partial_0 g) h) = \partial^* (f (\partial g) h)   
   \end{equation}
    (the actions on both sides depend on the derivations $\partial$ in (3.1) and $\partial_0$). This means that no matter the
   decomposition $-\Delta_u = \partial^* \partial$ or $-\Delta_u = \partial_0^* \partial_0$ is taken, we need to overcome exactly the same inequalities
   to obtain the (strong) Leibniz property.
 
  Now we can prove the following bimodule property of the norm $\|\cdot\|_{2, \partial}$ of the previous section.

  \begin{prop}
    For any $f,g$ and $h \in \ell^\infty_n,$
     $$ \|\partial^* (f (\partial g) h)\|_{2} \leq \|f\|_\infty \|g\|_\infty \|\partial^*\partial g \|_{2}.$$
  \end{prop}
  
  \begin{proof}
  
   Notice that (3.6) guarantees that it is enough to prove the inequality
    $$ \|\partial_0^* (f (\partial_0 g) h)\|_{2} \leq \|f\|_\infty \|g\|_\infty \|\partial_0^*\partial_0 g \|_{2}.$$
   Since $\Delta_u$ is an orthogonal projection in $\ell^2_n,$ for any $x \in \mathbb{R}^n,$
     $$  \|\Delta_u x\|_2 = \langle -\Delta_u x,x \rangle ^{1/2} = \|\partial_0 x\|_{\mathcal{H}},$$
    where we used formula (3.4).
  
  We observe that  $\|\partial_0^* A\|_{2} \leq \|A\|_\mathcal{H}$ for any $A \in M_n(\mathbb{R}).$ Indeed, for any $x \in \mathbb{R}^n,$
   the Cauchy--Schwarz inequality and the previous equality imply $$\langle \partial_0^* A , \partial_0^* A \rangle = (  A , \partial_0 \partial_0^* A )_\mathcal{H} 
   \leq  \| A\|_\mathcal{H}  \| \partial_0 \partial_0^* A \|_\mathcal{H} = \| A\|_\mathcal{H} \| \Delta_u \partial_0^* A \|_2 \leq \| A\|_\mathcal{H} \|\partial_0^* A \|_2 . $$
   Furthermore, \cite[Theorem 3.7]{CS} gives the inequality
   \begin{equation*}
    \|f (\partial_0 g) h\|_\mathcal{H} \leq \|f\|_\infty \|h\|_\infty \|\partial_0 g\|_\mathcal{H}   
   \end{equation*}
   for any $f, g$ and $h \in \ell^\infty_n.$
  
    Combining these observations, we get 
          \begin{eqnarray*}
       \begin{split}
        \|\partial_0^* (f (\partial_0 g) h)\|_{2}&\leq \|f (\partial_0 g) h\|_{\mathcal{H}} \\
                                           &\leq \|f\|_\infty \|h\|_\infty \|\partial_0 g\|_\mathcal{H}    \\                                           
                                           &=  \|f\|_\infty \|h\|_\infty \|\partial_0^* \partial_0 g\|_{2} ,
       \end{split}                                           
      \end{eqnarray*}   
      so the proof is complete.
       \end{proof}

  \subsubsection{Rieffel's non-commutative Riemann metric} In \cite{R1} Rieffel introduced the concept of the non-commutative
  Riemann metric that turns out to be a rich source of strongly Leibniz seminorms in finite dimensional $C^*$-algebras. 
  They provide us with an alternative way to obtain the Laplacian $\Delta_u$ as the divergence of a derivation.
  In short, a non-commutative Riemann
  metric is a normed first order differential calculus $(\Omega, \partial)$ over a $C^*$-algebra $\mathcal{A}$ with an $\mathcal{A}$-valued correspondence 
  $(\cdot, \cdot)_\mathcal{A}$ defined on $\Omega,$ see \cite[Section 3]{R1}. In particular, let us define the $\ell^\infty_n$-valued pre-inner product
  on the algebraic tensor product $\ell^\infty_n \otimes \ell^\infty_n$ by
   $$ (c \otimes d, a \otimes b)_{\ell^\infty_n} = bd \Gamma_{\Delta_u} (a,c),$$
  where $\Gamma_{\Delta_u}$ is the carr\'e-du-champ operator
   $$ \Gamma_{\Delta_u} (a,c) = a \Delta_u c  +  c \Delta_u a - \Delta_u(ac).$$
  Then $$ {1 \over n} \mathscr{E}_{\Delta_u}(f,g) = \mathbb{E} (\Gamma_{\Delta_u}(f,g) ) =  \mbox{Cov}_u(f,g), $$ where $\mathbb{E}$ and Cov$_u$ denotes the expected value 
  and the covariance with respect to the uniform probability measure.
  A positive bilinear form on $\ell^\infty_n \otimes \ell^\infty_n$ is given by
   $$ (c \otimes d, a \otimes b)_\mathcal{H} = \mathbb{E} (bd \Gamma_{\Delta_u} (a,c)). $$ It is simple to see that 
    \begin{eqnarray*}
        \begin{split}
          (c \otimes d, a \otimes b)_\mathcal{H} &= \mathbb{E}(bcd \Delta_u a  +  abd \Delta_u c - bd \Delta_u(ac)) \\
                                                 &= {1 \over 2n} (\mathscr{E}_{\Delta_u} (c, abd) + \mathscr{E}_{\Delta_u} (bcd, a)- \mathscr{E}_{\Delta_u}(bd, ac) ).                                                 
        \end{split}
    \end{eqnarray*}
  which essentially agrees with the bilinear form used by Cipriani and Sauvageot and leads to (3.5). 

   \section{An example} 
    
    The failure of the bimodule inequality (3.3) in the previous section suggests the following finite dimensional 
    example of a Leibniz seminorm that is not strongly Leibniz. 
    Such an example seems to have been unnoticed so far (see \cite[p. 54]{R2}).
    
    Let us define the Laplace matrix
      $$ \Delta_3 = \begin{bmatrix}
                        -2 & 1 & 1 \\
                         1 & -1 & 0 \\
                         1 & 0 & -1 
                    \end{bmatrix}.
    $$
    
           Then the  seminorm $$L(f) = \|\Delta_3f\|_\infty $$ defined on $\mathbb{R}^3$ is a Leibniz seminorm that is not strongly Leibniz.
    
         In fact, let us choose the vector $f =  (-0.1, 0.1, -0.2)^T.$ A direct calculation gives that the inequality $L(1/f) \leq \|1/f\|_\infty^2 L(f)$ does not hold. 
         For the Leibniz rule, let us consider the decomposition 
          $$ \Delta_3( fg) = \Pi(f)g + \Pi(g)f,$$ where
         $$ \Pi(x) = {1 \over 2} \begin{bmatrix}
                        -(2x_1 + x_2 + x_3) & x_1 + x_2  & x_1 + x_3 \\
                         x_1 + x_2 & -(x_1 + x_2) & 0 \\
                         x_1 + x_3 & 0 & -(x_1 + x_3) 
                    \end{bmatrix}. $$ 
         Then 
         $$ \|\Pi(f)g\|_\infty \leq \|f\|_\infty  \|\Delta_3 g\|_\infty \quad \mbox{and} \quad \|\Pi(g)f\|_\infty \leq \|g\|_\infty  \|\Delta_3 f\|_\infty.$$
         Indeed, without loss of generality, we can assume that $\|f\|_\infty = 1.$ The function $f \mapsto \|\Pi(f)g\|_\infty$ is convex 
         on the cube $[-1,1]^3,$ hence it attains its maximum if $f$ is in the vertex set $\{-1,1\}^3.$ In addition, 
          $$ \|\Pi(f)g\|_\infty =  \max ( |\varepsilon_{12}(g_1 - g_2) + \varepsilon_{13}(g_1 - g_3)|, |\varepsilon_{12}(g_1-g_2)|, |\varepsilon_{13}(g_1 - g_3)| ),$$
         where $\displaystyle \varepsilon_{ij} = {f_i + f_j \over 2} \in \{-1,0,1\}.$ Then it is straightforward to see that 
         $$ \|\Pi(f)g\|_\infty \leq \|\Delta_3 g\|_\infty =  \max ( |g_1 - g_2 + g_1 - g_3|, |g_1-g_2|, |g_1 - g_3| ).$$
         We can derive similarly the rest of the statement, which gives the requested result. 
         
         It would be interesting to know if the Leibniz inequality 
          $$ \|\Delta(fg)\| \leq \|f\|_\infty \|\Delta g\| + \|g\|_\infty \|\Delta f\|$$
         holds for every $n \times n$ Laplacian $\Delta.$ This would be a particular discrete version of the Kato--Ponce inequality studied intensively
         in PDEs, see e.g \cite{KP}, \cite{GO} and \cite{BL}

       \section{An application : the continuous case}

 In probability theory and statistics central moments and absolute central moments
 are primary objects which usually appear in estimates of probability distribution and their characteristic functions.        
 M. Rieffel proved that the standard deviation is a strongly Leibniz seminorm in commutative 
 and non-commutative probability spaces. He even extended
 these results to the case of matricial seminorms  on a unital $C^*$-algebra \cite{R2}.
 
 We are now in a position to prove the Leibniz inequality for higher order absolute moments of bounded real-valued random variables.
 
%
 Here is one of the main results of the paper.      

 \begin{thm}
    Let $(S, \mathcal{F}, \mu)$ be a probability space and $1\leq p < \infty.$  For any real $f$ and $g \in L^\infty(S, \mu),$ we have
 $$ \|fg-\mathbb{E}(fg)\|_p \leq \|g\|_{\infty}\|f-\mathbb{E}f\|_p + \|f\|_{\infty}\|g-\mathbb{E}g\|_p .$$
 \end{thm}
 
 \begin{proof}
  The statement is a corollary of Theorem 2.6 and the equivalence of Proposition 2.1 proved in \cite{BeL}.
 \end{proof}

 It would be interesting to have similar estimates in rearrangement invariant Banach function spaces. 

 \section*{Funding}

This work was supported by the Marie Curie Individual Fellowship, Project 'Moments' [653943]; and by the Hungarian Scientific Research Fund [K104206].
 
 \section*{Acknowledgement}
	
	The author wishes to thank Professor Koenraad Audenaert for stimulating and useful discussions.


\begin{thebibliography}{99}
     \bibitem{BeL} Á. Besenyei and Z. Léka, Leibniz seminorms in probability spaces, {\it J. Math. Anal. Appl.}, {\bf 429} (2015), 1178--1189.         
     \bibitem{B} R. Bhatia, {\it Matrix analysis}, Springer--Verlag New York, 1997. 
     \bibitem{BH} N. Bouleau and F. Hirsch, {\it Dirichlet Forms and Analysis on Wiener space}, de Gruyter Studies in Mathematics, De Gruyter, Berlin, 1991. 
     \bibitem{BL} J. Bourgain and D. Li, On an endpoint Kato--Ponce inequality, {\it Differential Integral Equations}, {\bf  27} (2014), 1037--1072.      
     \bibitem{C} A.-P.  Calderón, Spaces between $L^1$ and $L^\infty$ and the theorem of Marcinkiewicz, {\it Stud. Math.}, {\bf  26} (1966), 273--299.      
     \bibitem{AC} A. Connes, Compact metric spaces, Fredholm modules and hyperfiniteness, {\it Ergodic Theory and Dynamical Systems}, {\bf  9} (1989), 207--220.
     \bibitem{CS} F. Cipriani and J.-L. Sauvageot, Derivations as square roots of Dirichlet forms, {\it J. Func. Anal.}, {\bf 201} (2003), 78--120.
     \bibitem{Fuk} M. Fukushima, {\it Dirichlet Forms and Markov Processes}, North Holland Mathematical Library, 1980.
     \bibitem{GO} L. Grafakos and S. Oh, The Kato--Ponce inequality, {\it Comm. Partial Differential Equations}, {\bf 39} (2014), 1128--1157.     
     \bibitem{KP} T. Kato and G. Ponce, Commutator estimates and the Euler and Navier-Stokes equations, {\it Comm. Pure Appl. Math.}, {\bf 41} (1988), 891--907.
     \bibitem{KL} D. Kerr and H. Li, On Gromov-Hausdorff convergence for operator metric spaces, {\it J. Operator Theory}, {\bf 62} (2009), 83--109. 
     \bibitem{La} F. Latrémoli\`{e}re, Quantum metric spaces and the Gromov--Hausdorff propinquity, preprint, http://arxiv.org/pdf/1506.04341. 
     \bibitem{Y} R. Lyons and Y. Peres, {\it Probability on trees and networks}, preprint.
     \bibitem{K} J. Kigami, {\it  Analysis on Fractals}, Cambridge University Press, Cambridge, 2001. 
     \bibitem{Mi} B.S. Mityagin, An interpolation theorem for modular spaces (Russian), {\it Mat. Sb. (N.S.)}, {\bf  66} (1965), 473--482.          
     \bibitem{Pe} G.K. Pedersen, {\it  Analysis Now}, Springer--Verlag, 1989.      
     \bibitem{R0} M.A. Rieffel, Metrics on state spaces, {\it Doc. Math.}, {\bf 4} (1999), 559--600.
     \bibitem{R0c} M.A. Rieffel, Matrix algebras converge to the sphere for quantum Gromov--Hausdorff distance, Memoirs Amer. Math. Soc, 2004.
     \bibitem{RL} M.A. Rieffel, Leibniz seminorms for ``matrix algebras converge to the sphere``, {\it Quanta of Maths} {\bf 11}, Amer. Math. Soc., Providence, RI, 2010,  543--578.
     \bibitem{R1} M.A. Rieffel, Non-commutative resistance networks, {\it SIGMA Symmetry Integrability Geom. Methods Appl.},{ \bf 10} (2014), 2259--2274.     
     \bibitem{R2} M.A. Rieffel, Standard deviation is a strongly Leibniz seminorm, {\it New York J. Math.}, {\bf 20} (2014), 35--56.
     \bibitem{JLS} J.--L. Sauvageot, Quantum Dirichlet forms, differential calculus and semigroups, 
              Quantum probability and applications, V (Heidelberg, 1988), 334–346, Lecture Notes in Math., 1442, Springer, Berlin, 1990. 
     \bibitem{BS} B. Simon, {\it  Convexity: An Analytic Viewpoint}, Cambridge University Press, 2011.      
     \bibitem{NW} N. Weaver, Lipschitz algebras and derivations of von Neumann algebras, {\it J. Funct. Anal.}, 139 (1996), 261--300.
     \bibitem{NW2} N. Weaver, {\it Lipschitz algebras}, World Scientific Publishing Co., Inc., River Edge, NJ, 1999.  
     \bibitem{Wu} W. Wu, Quantized Gromov-Hausdorff distance, {\it J. Funct. Anal.}, 238 (2006), 58--98.  
   \end{thebibliography}
\end{document}